\documentclass[11pt,twoside]{article}
\usepackage{fancyhead}
\usepackage{psfrag}
\usepackage{epsfig}
\usepackage{graphicx}
\usepackage{amsmath}
\usepackage{amssymb}
\usepackage{amsthm}
\usepackage{subfigure}
\usepackage{cite}
\usepackage[latin1]{inputenc}

\newtheorem{theorem}{Theorem}
\newtheorem{conjecture}{Conjecture}

\begin{document}

\date{}

\noindent \textbf{\LARGE New sizes of complete arcs in
$PG(2,q)$} \thispagestyle{fancyplain} \setlength\partopsep
{0pt} \flushbottom

\vspace*{5mm}

\noindent \textsc{Alexander A.
Davydov}\hfill\texttt{adav@iitp.ru}\newline {\small Institute
for Information Transmission Problems, Russian Academy of
Sciences,\\ Bol'shoi Karetnyi per. 19, GSP-4, Moscow, 127994,
Russia}\newline \textsc{Giorgio Faina} \hfill
\texttt{faina@dipmat.unipg.it}\newline \textsc{Stefano
Marcugini} \hfill \texttt{gino@dipmat.unipg.it}\newline
\textsc{Fernanda Pambianco}\hfill
\texttt{fernanda@dipmat.unipg.it}\newline {\small Dipartimento
di Matematica e Informatica, Universit\`{a} degli Studi di
Perugia, Via Vanvitelli~1, Perugia, 06123, Italy}

\medskip

\begin{center}
\parbox{11,8cm}{\footnotesize\textbf{Abstract.}
New upper bounds on the smallest size $t_{2}(2,q)$ of a
complete arc in the projective plane $PG(2,q)$ are obtained for
$853\leq q\leq 2879$ and $q=3511,4096$,
$4523,5003,5347,5641,5843,6011$. For $q\leq 2377$ and
$q=2401,2417,2437$, the relation $t_{2}(2,q)<4.5\sqrt{q}$
holds. The bounds are obtained by finding of new small complete
arcs with the help of computer search using randomized greedy
algorithms. Also new sizes of complete arcs are presented.}
\end{center}

\baselineskip=0.9\normalbaselineskip

\section{Introduction}

Let $PG(2,q)$ be the projective plane over the Galois field
$F_{q}$. An $n$-arc is a set of $n$ points no $3$ of which are
collinear. An $n$-arc is called complete if it is not contained
in an $(n+1)$-arc of $PG(2,q)$. Surveys of results on arcs can
be found in \cite{librohir,surveyhir}. In
\cite{surveyhir} the close relationship between the theory of
complete $n$-arcs, coding theory and mathematical statistics is
presented. In particular a complete arc in a plane $PG(2,q),$
points of which are treated as 3-dimensional $q$-ary columns,
defines a parity check matrix of a $q$-ary linear code with
codimension 3, Hamming distance 4 and covering radius 2. Arcs
can be interpreted as linear maximum distance separable (MDS)
codes and they are related to optimal coverings arrays
\cite{Hartman-Haskin} and to superregular matrices \cite{Keri}.

One of the main problems in the study of projective planes,
which is also of interest in Coding Theory, is the
determination of the spectrum of possible sizes of complete
arcs. Especially the problem of determining $t_{2}(2,q)$, the
size of the smallest complete arc in $PG(2,q)$, is interesting.

In Section 2 we give upper bounds on $t_{2}(2,q)$  for $853\leq
q\leq 2879$ and $q=3511,4096,4523,5003,5347,5641,5843,6011$.
These bounds are new for almost all $q$. For  $q\leq 2377$ and
$q=2401,2417,2437$, the relation $t_{2}(2,q)<4.5\sqrt{q}$
holds. For smaller $q$ slightly smaller bounds hold. The upper
bounds have been obtained by finding of new small complete arcs
with the help of the randomized greedy algorithms  described in
\cite[Sect.\thinspace 2]{DFMP-Plane},\cite[Sect.\thinspace
2]{DMP-JG2004}.

In Section 3 we present new sizes of complete arcs in $PG(2,q)$
 with \linebreak $169\le q\le 349$ and $q=1013,2003$.

\section{Small complete $k$-arcs in $PG(2,q)$, $
853\le q\le 2879$}

In the plane $PG(2,q)$, we denote  $\overline{t}_{2}(2,q)$ the
smallest \emph{known} size of complete arcs. For $q\le 841$,
the values of $\overline{t}_{2}(2,q)<4\sqrt{q}$  are collected
in \cite[Tab.\thinspace1]{DFMP-JG}.

In Tables 1 and 2, the values of
 $\overline{t}_{2}(2,q)$ for $853\leq
q\leq 2879$ and $q=3511,4096,4523,5003,5347,5641,5843,6011$ are
given. We denote\newline $ A_{q}=\left\lfloor
4.5\sqrt{q}-\overline{t}_{2}(2,q)\right\rfloor $, $B_{q}$ a
superior approximation of $\overline{t}_{2}(2,q)/\sqrt{q}$.
Also, $ C_{q}=\left\lfloor
5\sqrt{q}-\overline{t}_{2}(2,q)\right\rfloor $. For all $q$ in
Table~1 and $q=2401,2417,2437$ in Table 2, it holds that
$\overline{t} _{2}(2,q)<4.5 \sqrt{q}$.

In \cite{G}, complete $k$-arcs are obtained with $
k=4(\sqrt{q}-1)$, $q=p^{2}$ odd, $q\le 1681$ or $q=2401$. For
even $q=2^{h}$, $10\le h\le 15$, the smallest known sizes of
complete $k$ -arcs in $PG(2,q)$ are obtained in
\cite{DGMP-JCD}, see also \cite[p.\thinspace 35]{DFMP-JG}. They
are as follows: $ k=124,201,307,461,665,993,$ for
$q=2^{10},2^{11},2^{12},2^{13},2^{14},2^{15},$ respectively.
Also, $6{(\sqrt{q}-1)}${-arcs} in $\mathrm{PG}(2,q),$ $
q=4^{2h+1},$ are constructed in \cite{DGMP-Innov}; for $h\leq
4$ it is proved that they are complete. It gives a complete
3066-arc in $\mathrm{PG} (2,2^{18}).$ In Tables 1 and 2, we use
the results of \cite{G,DGMP-JCD} for
$q=31^{2},37^{2},41^{2},7^{4},2^{10},2^{11}$.

The rest of sizes $k$ for small complete arcs in Tables 1 and 2
is obtained in this work by computer search with the help of
the randomized greedy algorithms.

From Tables 1 and 2, we obtain Theorems \ref{Th_Plane_/} and
\ref{Th_Plane_-}.

\begin{theorem}
\label{Th_Plane_/} In $PG(2,q),$
\[
\begin{array}{ccl}
t_{2}(2,q) &<&4.5\sqrt{q}~\mbox{ for }q\leq 2377,\,q=2401,2417,2437; \\
t_{2}(2,q) &<&4.2\sqrt{q}~\mbox{ for }q\leq 1163,\,q=1181,1193,1369,1681,2401; \\
t_{2}(2,q) &<&4.3\sqrt{q}~\mbox{ for }q\leq
1451,\,q=1459,1471,1481,1483,1493,1499,1511,\\
&&\phantom{4.4\sqrt{q}~ \mbox{ for
}q\leq 1823,\, q=\,\,}1681,2401; \\
t_{2}(2,q) &<&4.4\sqrt{q}~\mbox{ for
}q\leq 1849,\,q=1867,1889,1901,1907,1913,1949,1993,\\
&&\phantom{4.4\sqrt{q}~ \mbox{ for }q\leq 1823,\, q=\,\,}2401.
\end{array}
\]
\end{theorem}

\begin{theorem}
\label{Th_Plane_-} In $PG(2,q),$
\begin{center}
$
\begin{array}{ccl}
t_{2}(2,q) &<&4.5\sqrt{q}\quad \phantom{-1}\mbox{ for }q\leq 2377,\,q=2401,2417,2437;
\label{ARCS} \\
t_{2}(2,q) &<&4.5\sqrt{q}-10\mbox{ for }q\leq
1163,\,q=1181,1187,1193,1223,1237,1249,   \\
&&\phantom{4.5\sqrt{q}-10\mbox{ for }q\leq
1117,\,q=\,\,}1369,1681,2401;\\
t_{2}(2,q) &<&4.5\sqrt{q}-8\,~\mbox{ for }q\leq
1423,\,q=1429,1433,1439,1447,1451,1471,  \\
&&\phantom{4.5\sqrt{q}-8\,~ \mbox{ for } q\leq 1369,\, q=\,\,}
1481,1483,1499,1511,1681,2401;   \\
t_{2}(2,q) &<&4.5\sqrt{q}-6\,~\mbox{ for
}q\leq 1693,\, q=1699,1709,1747,1783,2401;   \\
t_{2}(2,q) &<&4.5\sqrt{q}-3\,~\mbox{ for }q\leq 2003,\, q=2017,2027,2401.
\end{array}
$
\end{center}
\end{theorem}
Our methods allow us to obtain small arcs in
 $PG(2,q)$ for $q\le6011$, using our present computers.
We plan to write on these arcs sizes in a
 journal paper.

Let $c$ be a constant independent of $q$. Let
$t(\mathcal{P}_{q})$ be the size of the smallest complete arc
in any projective plane $\mathcal{P}_{q}$ of order $q$. In
\cite{KV}, for sufficiently large \newpage

Table 1. The smallest known sizes
$\overline{t}_{2}=\overline{t} _{2}(2,q)<4.5 \sqrt{q}$ of
complete arcs in planes $PG(2,q)$. $ A_{q}=\left\lfloor
4.5\sqrt{q}-\overline{t}_{2}(2,q)\right\rfloor$, $B_{q}>
\overline{t}_{2}(2,q)/ \sqrt{q}$
\begin{center}
$
\begin{array}{|r|c|c|c||c|c|r|c||c|c|r|c|}
\hline
q & \overline{t}_{2} & A_{q} & B_{q} & q & \overline{t}_{2} & A_{q} & B_{q}
& q & \overline{t}_{2} & A_{q} & B_{q} \\ \hline
 853 & 118 & 13 & 4.05 & 1087 & 137 & 11 & 4.16 & 1327 & 155 & 8 & 4.26 \\
 857 & 119 & 12 & 4.07 & 1091 & 138 & 10 & 4.18 & 1331 & 155 & 9 & 4.25 \\
 859 & 119 & 12 & 4.07 & 1093 & 138 & 10 & 4.18 & 1361 & 157 & 9 & 4.26 \\
 863 & 119 & 13 & 4.06 & 1097 & 138 & 11 & 4.17 & 1367 & 158 & 8 & 4.28 \\
 877 & 120 & 13 & 4.06 & 1103 & 138 & 11 & 4.16 & 1369 & 144 & 22& 3.90 \\
 881 & 121 & 12 & 4.08 & 1109 & 138 & 11 & 4.15 & 1373 & 158 & 8 & 4.27 \\
 883 & 121 & 12 & 4.08 & 1117 & 140 & 10 & 4.19 & 1381 & 159 & 8 & 4.28 \\
 887 & 121 & 13 & 4.07 & 1123 & 139 & 11 & 4.15 & 1399 & 160 & 8 & 4.28 \\
 907 & 123 & 12 & 4.09 & 1129 & 140 & 11 & 4.17 & 1409 & 160 & 8 & 4.27 \\
 911 & 123 & 12 & 4.08 & 1151 & 142 & 10 & 4.19 & 1423 & 161 & 8 & 4.27 \\
 919 & 124 & 12 & 4.10 & 1153 & 142 & 10 & 4.19 & 1427 & 162 & 7 & 4.29 \\
 929 & 125 & 12 & 4.11 & 1163 & 143 & 10 & 4.20 & 1429 & 161 & 9 & 4.26 \\
 937 & 126 & 11 & 4.12 & 1171 & 144 & 9  & 4.21 & 1433 & 161 & 9 & 4.26 \\
 941 & 126 & 12 & 4.11 & 1181 & 144 & 10 & 4.20 & 1439 & 161 & 9 & 4.25 \\
 947 & 127 & 11 & 4.13 & 1187 & 145 & 10 & 4.21 & 1447 & 162 & 9 & 4.26 \\
 953 & 127 & 11 & 4.12 & 1193 & 145 & 10 & 4.20 & 1451 & 163 & 8 & 4.28 \\
 961 & 120 & 19 & 3.88 & 1201 & 146 & 9  & 4.22 & 1453 & 164 & 7 & 4.31 \\
 967 & 128 & 11 & 4.12 & 1213 & 147 & 9  & 4.23 & 1459 & 164 & 7 & 4.30 \\
 971 & 128 & 12 & 4.11 & 1217 & 147 & 9  & 4.22 & 1471 & 164 & 8 & 4.28 \\
 977 & 129 & 11 & 4.13 & 1223 & 147 & 10 & 4.21 & 1481 & 164 & 9 & 4.27 \\
 983 & 129 & 12 & 4.12 & 1229 & 148 & 9  & 4.23 & 1483 & 165 & 8 & 4.29 \\
 991 & 130 & 11 & 4.13 & 1231 & 148 & 9  & 4.22 & 1487 & 166 & 7 & 4.31 \\
 997 & 130 & 12 & 4.12 & 1237 & 148 & 10 & 4.21 & 1489 & 166 & 7 & 4.31 \\
1009 & 132 & 10 & 4.16 & 1249 & 149 & 10 & 4.22 & 1493 & 166 & 7 & 4.30 \\
1013 & 131 & 12 & 4.12 & 1259 & 150 & 9  & 4.23 & 1499 & 166 & 8 & 4.29 \\
1019 & 132 & 11 & 4.14 & 1277 & 151 & 9  & 4.23 & 1511 & 166 & 8 & 4.28 \\
1021 & 132 & 11 & 4.14 & 1279 & 151 & 9  & 4.23 & 1523 & 168 & 7 & 4.31 \\
1024 & 124 & 20 & 3.88 & 1283 & 152 & 9  & 4.25 & 1531 & 169 & 7 & 4.32 \\
1031 & 132 & 12 & 4.12 & 1289 & 152 & 9  & 4.24 & 1543 & 169 & 7 & 4.31 \\
1033 & 133 & 11 & 4.14 & 1291 & 152 & 9  & 4.24 & 1549 & 170 & 7 & 4.32 \\
1039 & 134 & 11 & 4.16 & 1297 & 153 & 9  & 4.25 & 1553 & 170 & 7 & 4.32 \\
1049 & 134 & 11 & 4.14 & 1301 & 153 & 9  & 4.25 & 1559 & 170 & 7 & 4.31 \\
1051 & 135 & 10 & 4.17 & 1303 & 153 & 9  & 4.24 & 1567 & 171 & 7 & 4.32 \\
1061 & 135 & 11 & 4.15 & 1307 & 153 & 9  & 4.24 & 1571 & 171 & 7 & 4.32 \\
1063 & 136 & 10 & 4.18 & 1319 & 154 & 9  & 4.25 & 1579 & 172 & 6 & 4.33 \\
1069 & 136 & 11 & 4.16 & 1321 & 154 & 9  & 4.24 & 1583 & 172 & 7 & 4.33 \\
\hline
\end{array}
$
\end{center}

\newpage

Table 1 (continue). The smallest known sizes
$\overline{t}_{2}=\overline{t} _{2}(2,q)<4.5 \sqrt{q}$ of
complete arcs in planes $PG(2,q)$. $ A_{q}=\left\lfloor
4.5\sqrt{q}-\overline{t}_{2}(2,q)\right\rfloor$, $B_{q}>
\overline{t}_{2}(2,q)/ \sqrt{q}$

\begin{center}
$\renewcommand{\arraystretch}{0.95}
\begin{array}{|c|c|r|c||c|c|c|c||c|c|c|c|}
\hline
q & \overline{t}_{2} & A_{q} & B_{q} & q & \overline{t}_{2} & A_{q} & B_{q}
& q & \overline{t}_{2} & A_{q} & B_{q} \\ \hline
1597 & 173 & 6 & 4.33 & 1867 & 190 & 4 & 4.40 & 2129 & 206 & 1 & 4.47 \\
1601 & 173 & 7 & 4.33 & 1871 & 191 & 3 & 4.42 & 2131 & 206 & 1 & 4.47 \\
1607 & 174 & 6 & 4.35 & 1873 & 191 & 3 & 4.42 & 2137 & 206 & 2 & 4.46 \\
1609 & 174 & 6 & 4.34 & 1877 & 191 & 3 & 4.41 & 2141 & 206 & 2 & 4.46 \\
1613 & 174 & 6 & 4.34 & 1879 & 191 & 4 & 4.41 & 2143 & 207 & 1 & 4.48 \\
1619 & 174 & 7 & 4.33 & 1889 & 191 & 4 & 4.40 & 2153 & 207 & 1 & 4.47 \\
1621 & 174 & 7 & 4.33 & 1901 & 191 & 5 & 4.39 & 2161 & 207 & 2 & 4.46 \\
1627 & 175 & 6 & 4.34 & 1907 & 192 & 4 & 4.40 & 2179 & 209 & 1 & 4.48 \\
1637 & 176 & 6 & 4.35 & 1913 & 192 & 4 & 4.39 & 2187 & 209 & 1 & 4.47 \\
1657 & 177 & 6 & 4.35 & 1931 & 194 & 3 & 4.42 & 2197 & 208 & 2 & 4.44 \\
1663 & 177 & 6 & 4.35 & 1933 & 194 & 3 & 4.42 & 2203 & 209 & 2 & 4.46 \\
1667 & 177 & 6 & 4.34 & 1949 & 194 & 4 & 4.40 & 2207 & 210 & 1 & 4.48 \\
1669 & 177 & 6 & 4.34 & 1951 & 195 & 3 & 4.42 & 2209 & 210 & 1 & 4.47 \\
1681 & 160 & 24 &3.91 & 1973 & 196 & 3 & 4.42 & 2213 & 210 & 1 & 4.47 \\
1693 & 179 & 6 & 4.36 & 1979 & 196 & 4 & 4.41 & 2221 & 210 & 2 & 4.46 \\
1697 & 180 & 5 & 4.37 & 1987 & 197 & 3 & 4.42 & 2237 & 211 & 1 & 4.47 \\
1699 & 179 & 6 & 4.35 & 1993 & 196 & 4 & 4.40 & 2239 & 211 & 1 & 4.46 \\
1709 & 180 & 6 & 4.36 & 1997 & 198 & 3 & 4.44 & 2243 & 211 & 2 & 4.46 \\
1721 & 181 & 5 & 4.37 & 1999 & 198 & 3 & 4.43 & 2251 & 212 & 1 & 4.47 \\
1723 & 181 & 5 & 4.37 & 2003 & 198 & 3 & 4.43 & 2267 & 213 & 1 & 4.48 \\
1733 & 182 & 5 & 4.38 & 2011 & 199 & 2 & 4.44 & 2269 & 213 & 1 & 4.48 \\
1741 & 182 & 5 & 4.37 & 2017 & 199 & 3 & 4.44 & 2273 & 214 & 0 & 4.49 \\
1747 & 182 & 6 & 4.36 & 2027 & 199 & 3 & 4.43 & 2281 & 214 & 0 & 4.49 \\
1753 & 183 & 5 & 4.38 & 2029 & 200 & 2 & 4.45 & 2287 & 215 & 0 & 4.50 \\
1759 & 183 & 5 & 4.37 & 2039 & 201 & 2 & 4.46 & 2293 & 215 & 0 & 4.49 \\
1777 & 184 & 5 & 4.37 & 2048 & 201 & 2 & 4.45 & 2297 & 215 & 0 & 4.49 \\
1783 & 183 & 7 & 4.34 & 2053 & 201 & 2 & 4.44 & 2309 & 215 & 1 & 4.48 \\
1787 & 185 & 5 & 4.38 & 2063 & 202 & 2 & 4.45 & 2311 & 216 & 0 & 4.50 \\
1789 & 185 & 5 & 4.38 & 2069 & 202 & 2 & 4.45 & 2333 & 217 & 0 & 4.50 \\
1801 & 186 & 4 & 4.39 & 2081 & 203 & 2 & 4.45 & 2339 & 217 & 0 & 4.49 \\
1811 & 187 & 4 & 4.40 & 2083 & 203 & 2 & 4.45 & 2341 & 217 & 0 & 4.49\\
1823 & 187 & 5 & 4.38 & 2087 & 203 & 2 & 4.45 & 2347 & 218 & 0 & 4.50 \\
1831 & 188 & 4 & 4.40 & 2089 & 203 & 2 & 4.45 & 2351 & 218 & 0 & 4.50 \\
1847 & 189 & 4 & 4.40 & 2099 & 204 & 2 & 4.46 & 2357 & 218 & 0 & 4.50 \\
1849 & 189 & 4 & 4.40 & 2111 & 205 & 1 & 4.47 & 2371 & 218 & 1 & 4.48 \\
1861 & 190 & 4 & 4.41 & 2113 & 205 & 1 & 4.46 & 2377 & 219 & 0 & 4.50\\
\hline
\end{array}
$
\end{center}
\newpage
\noindent Table 2. The smallest known sizes
$\overline{t}_{2}=\overline{t} _{2}(2,q)<5 \sqrt{q}$ of
complete arcs in planes $PG(2,q)$. $ A_{q}=\left\lfloor
4.5\sqrt{q}-\overline{t}_{2}(2,q)\right\rfloor$, $B_{q}>
\overline{t}_{2}(2,q)/ \sqrt{q}$, $ C_{q}=\left\lfloor
5\sqrt{q}-\overline{t}_{2}(2,q)\right\rfloor $
\begin{center}
$
\begin{array}{|c|c|c|c|c||c|c|c|c||c|c|c|c|}
\hline
q & \overline{t}_{2} & A_{q} &C_{q}  &B_{q} &q & \overline{t}_{2} & C_{q}   & B_{q}
& q & \overline{t}_{2} & C_{q} & B_{q} \\ \hline
2381&220&  &23&4.51&2551&229&23&4.54&2713&237&23&4.56\\
2383&220&  &24&4.51&2557&229&23&4.53&2719&238&22&4.57\\
2389&220&  &24&4.51&2579&230&23&4.53&2729&238&23&4.56\\
2393&221&  &23&4.52&2591&231&23&4.54&2731&238&23&4.56\\
2399&221&  &23&4.52&2593&231&23&4.54&2741&239&22&4.57\\
2401&192&28&53&3.92&2609&232&23&4.55&2749&239&23&4.56\\
2411&221&  &24&4.51&2617&233&22&4.56&2753&239&23&4.56\\
2417&221&0 &24&4.50&2621&233&22&4.56&2767&241&22&4.59\\
2423&222&  &24&4.51&2633&232&24&4.53&2777&241&22&4.58\\
2437&222&0 &24&4.50&2647&234&23&4.55&2789&241&23&4.57\\
2441&223&  &24&4.52&2657&233&24&4.53&2791&242&22&4.59\\
2447&223&  &24&4.51&2659&233&24&4.52&2797&241&23&4.56\\
2459&224&  &23&4.52&2663&235&23&4.56&2801&242&22&4.58\\
2467&224&  &24&4.51&2671&236&22&4.57&2803&242&22&4.58\\
2473&225&  &23&4.53&2677&236&22&4.57&2809&242&23&4.57\\
2477&225&  &23&4.53&2683&236&22&4.56&2819&242&23&4.56\\
2503&227&  &23&4.54&2687&236&23&4.56&2833&243&23&4.57\\
2521&227&  &24&4.53&2689&236&23&4.56&2837&244&22&4.59\\
2531&227&  &24&4.52&2693&237&22&4.57&2843&244&22&4.58\\
2539&228&  &23&4.53&2699&237&22&4.57&2851&244&22&4.57\\
2543&228&  &24&4.53&2707&237&23&4.56&2857&245&22&4.59\\
2549&229&  &23&4.54&2711&237&23&4.56&2861&245&22&4.59\\
                            &&&&&&&&&2879&245&23&4.57\\                \hline
\end{array}
$
\end{center}

\noindent $q$, it is proved that $t(\mathcal{P}_{q})\le
\sqrt{q}\log^{c}q$, $c=300$. The logarithm basis is not noted
as the estimate is asymptotic. For definiteness, we use the
binary logarithms. We introduce $D_{q}(c)$ and
$\overline{D}_{q}(c)$ as follows:
$$
t_{2}(2,q)=D_{q}(c)\sqrt{q}\log_{2}^{c}q, ~~
\overline{t}_{2}(2,q)=\overline{D}_{q}(c)\sqrt{q}\log_{2}^{c}q.
$$

From Tables 1, 2 and \cite[Tab.\thinspace 1]{DFMP-JG}, we
obtain Observation~1.

\textbf{Observation 1.} Let $173\le q\le 2879$, $q\ne 5^{4},3^{6},29^{2},31^{2},2^{10},37^{2},41^{2},7^{4}$. Then\\
\textbf{(i)} $0.45>\overline{D}_{q}(1)>0.397$. Also,
$0.428>\overline{D}_{q}(1)$ if $467\le q$;
$0.415>\overline{D}_{q}(1)$ if $1013\le q$;
$0.41>\overline{D}_{q}(1)$ if $1399\le q$;
 $0.405>\overline{D}_{q}(1)$ if $1889\le
q$. So, $\overline{D}_{q}(1)$ has a tendency to decreasing.\\
\textbf{(ii)} $1.202<\overline{D}_{q}(\frac{1}{2})<1.355$.
Also, $\overline{D}_{q}(\frac{1}{2})<1.27$ if $q\le 443$;
$\overline{D}_{q}(\frac{1}{2})<1.32$ if $q\le 1291$;
$\overline{D}_{q}(\frac{1}{2})<1.325$ if $q\le 1327$;
 $\overline{D}_{q}(\frac{1}{2})<1.335$ if
$q\le 1801$. So, $\overline{D}_{q}(\frac{1}{2})$ has a tendency to increasing.\\
\textbf{(iii)} $0.720<\overline{D}_{q}(0.75)<0.743$. The values
of $\overline{D}_{q}(0.75)$ oscillate about the average value
$0.73331$. It holds that
\begin{eqnarray}\label{eq2_Dq75}
\begin{array}{cc}
 0.720<\overline{D}_{q}(0.75)<0.743 &\mbox{if }173\le q\le997 , \smallskip\\
0.727<\overline{D}_{q}(0.75)<0.741 &\mbox{ if }1009\le q\le1999,\smallskip\\
0.729<\overline{D}_{q}(0.75)<0.738 &\mbox{ if } 2003\le q\le2879 .\medskip
\end{array}
\end{eqnarray}
Moreover, let
$$
\widehat{t}_{2}(2,q)=0.73331\sqrt{q}\log_{2}^{0.75}q,\quad
\overline{\Delta}_{q}= \overline{t}_{2}(2,q) -\widehat{t}_{2}(2,q),\quad
\overline{P}_{q}=\frac{100\overline{\Delta}_{q}}{\overline{t}_{2}(2,q)}\%.
$$
It holds that
\begin{eqnarray}\label{eq2_dlt}
-1.86\le\overline{\Delta}_{q}\le1.23.
\end{eqnarray}

\begin{eqnarray}\label{eq2_percent}
\begin{array}{cc}
 -1.73\%<\overline{P}_{q}<1.31\% &\mbox{if }173\le q\le997 , \smallskip\\
-0.80\%<\overline{P}_{q}<0.93\% &\mbox{ if }1009\le q\le1999,\smallskip\\
-0.53\%<\overline{P}_{q}<0.54\% &\mbox{ if } 2003\le q\le2879 .\medskip
\end{array}
\end{eqnarray}

 In other words,
 $\widehat{t}_{2}(2,q)=0.73331\sqrt{q}\log_{2}^{0.75}$ can be treated as a
 \textbf{predicted} value of $t_{2}(2,q)$. Then
 $\overline{\Delta}_{q}$ is the difference between the smallest
 known size $\overline{t}_{2}(2,q)$ of complete arcs and the predicted value. Finally,
 $\overline{P}_{q}$ is this difference in percentage terms of
 the smallest known size.

By (\ref{eq2_dlt}),(\ref{eq2_percent}), the magnitude of the
difference $\overline{\Delta}_{q}$ is smaller than two. The
magnitude of the percentage value $\overline{P}_{q}$ is smaller
than two for $q<1000$ and smaller than one for $q>1000$. The
region of $\overline{P}_{q}$   is decreasing with growth
of~$q$. Also, by (\ref{eq2_Dq75}), the region of
$\overline{D}_{q}(0.75)$   is decreasing with growth of $q$.

The graphs of values of $\overline{D}_{q}(0.75)$,
$\overline{\Delta}_{q}$, and $\overline{P}_{q}$ are shown on
Figures 1-3.
\begin{figure}[h!]
\begin{center}
\epsfig{file=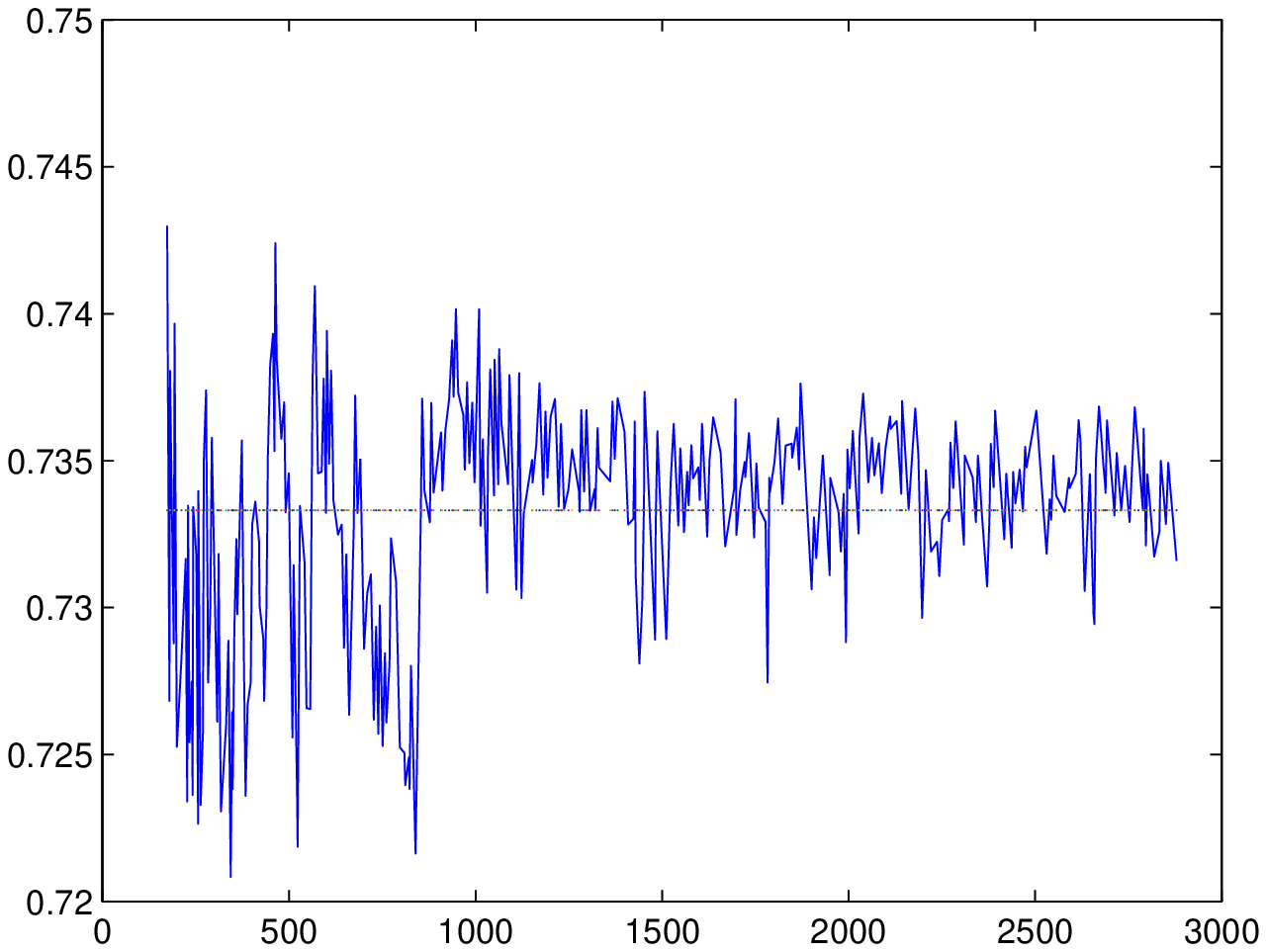,width=14.5cm}
\caption{The values of
$\overline{D}_{q}(0.75)=\frac{\overline{t}_{2}(2,q)}{\sqrt{q}\log_{2}^{0.75}q}$, $173\le q\le 2879$,
 $q\ne 5^{4},3^{6},29^{2},$ $31^{2},2^{10},37^{2},41^{2},7^{4}$}
\end{center}
\end{figure}

\begin{figure}[h!]
\begin{center}
\epsfig{file=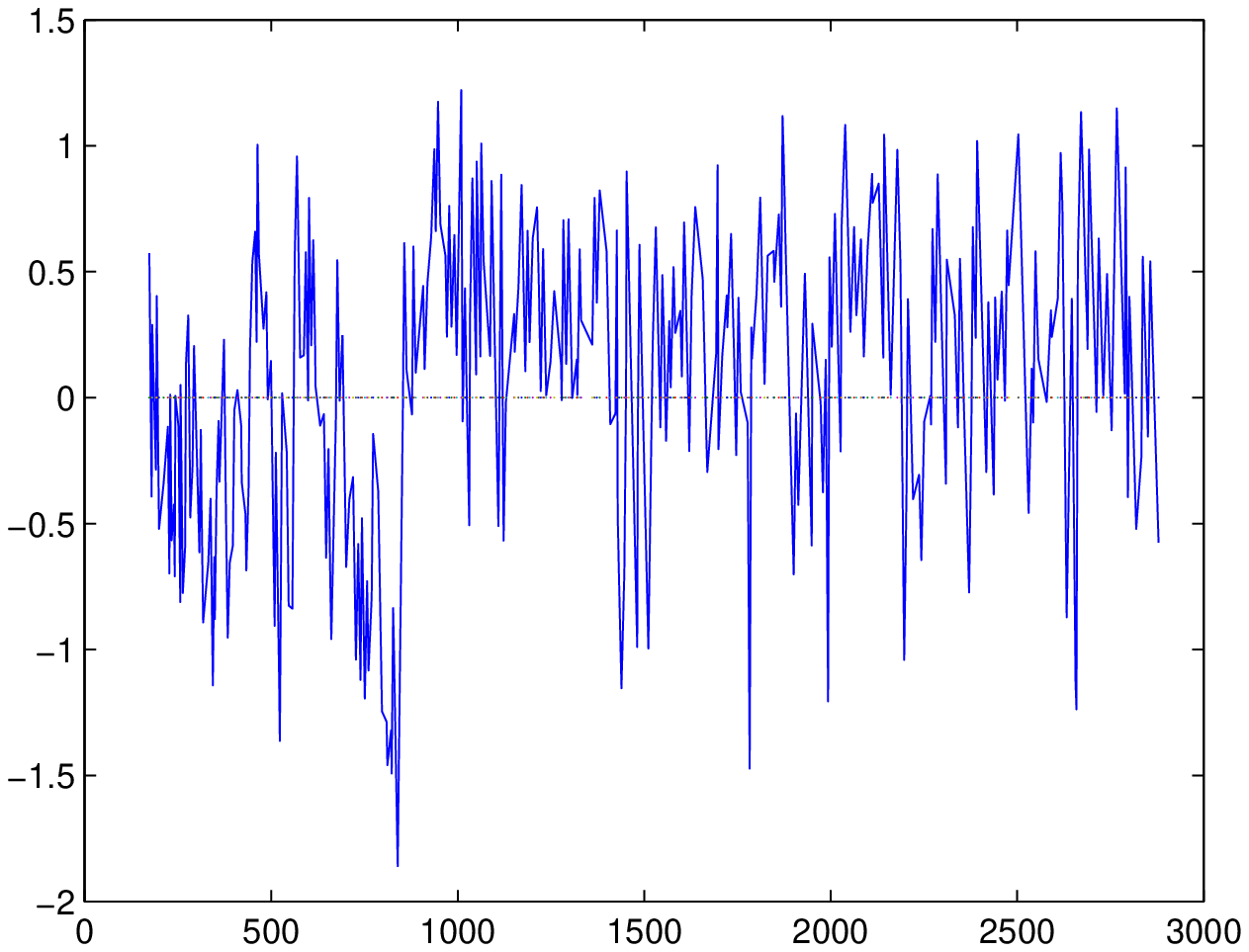,width=14.5cm}
\caption{The values of
$\overline{\Delta}_{q}= \overline{t}_{2}(2,q) -0.73331\sqrt{q}\log_{2}^{0.75}q $, $173\le q\le 2879$, \newline
 $q\ne 5^{4},3^{6},29^{2},31^{2},2^{10},37^{2},41^{2},7^{4}$}
\end{center}
\end{figure}

\begin{figure}[h!]
\begin{center}
\epsfig{file=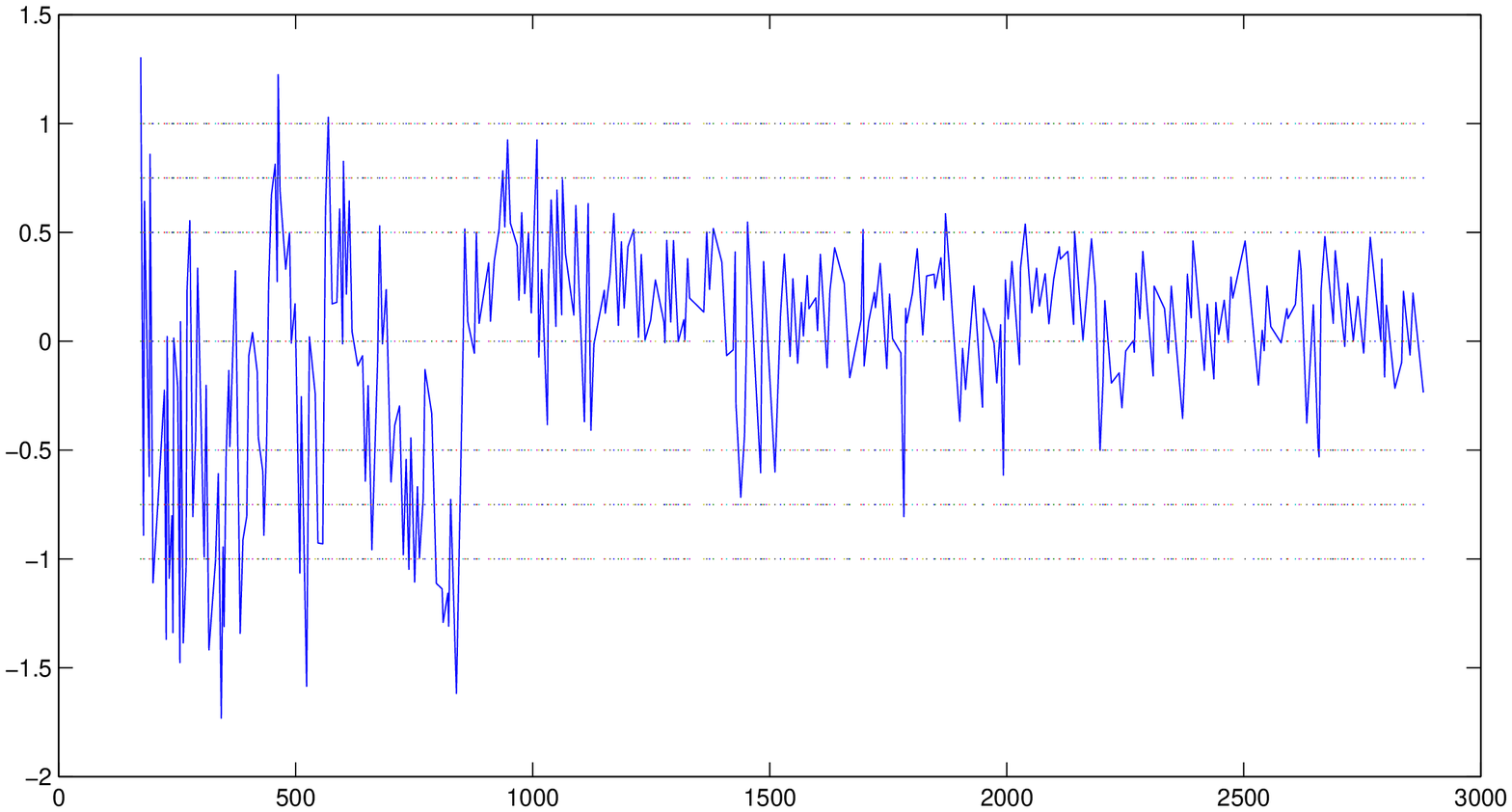,width=14.5cm}
\caption{The values of $\overline{P}_{q}=\frac{100\overline{\Delta}_{q}}{\overline{t}_{2}(2,q)}\%$, $173\le q\le 2879$,
  $q\ne 5^{4},3^{6},29^{2},31^{2},$ $2^{10},37^{2},41^{2},7^{4}$}
\end{center}
\end{figure}
Examples for great $q$ are given in Table 3.\medskip

\noindent Table 3. The smallest known sizes
$\overline{t}_{2}=\overline{t} _{2}(2,q)<5 \sqrt{q}$ of
complete arcs in planes $PG(2,q)$ with great $q$.  $B_{q}>
\overline{t}_{2}(2,q)/ \sqrt{q}$, $ C_{q}=\left\lfloor
5\sqrt{q}-\overline{t}_{2}(2,q)\right\rfloor $
\begin{center}
$\renewcommand{\arraystretch}{1.1}
\begin{array}{|@{\,}c@{\,}|@{\,}c@{\,}|@{\,}c@{\,}|@{\,}c@{\,}|@{\,}c@{\,}|@{\,}c@{\,}|
@{\,}c@{\,}||@{\,}c@{\,}|@{\,}c@{\,}|@{\,}c@{\,}|@{\,}c@{\,}|@{\,}c@{\,}|@{\,}c@{\,}|@{\,}c@{\,}|}
\hline
q & \overline{t}_{2} & C_{q}  &B_{q} &\overline{D}_{q}(1)&\overline{D}_{q}(\frac{1}{2})&\overline{D}_{q}(\frac{3}{4})&
q & \overline{t}_{2} & C_{q}  &B_{q} &\overline{D}_{q}(1)&\overline{D}_{q}(\frac{1}{2})&\overline{D}_{q}(\frac{3}{4})\\ \hline
3511&278&18&4.70&0.398&1.367&0.7380&5347&353&12&4.83&0.390&1.372&0.7312\\
4096&302&18&4.72&0.393&1.362&0.7319&5641&364&11&4.85&0.389&1.373&0.7307\\
4523&322&14&4.79&0.394&1.374&0.7360&5843&373&9&4.88&0.390&1.379&0.7335\\
5003&341&12&4.83&0.392&1.375&0.7345&6011&377&10&4.87&0.387&1.372&0.7291\\ \hline
\end{array}
$
\end{center}
The examples confirm Observation 1. So, along with $B_{q}$, the
values $\overline{D}_{q}(c)$, in particular with $c=0.75$, can
be useful for estimates of complete arcs sizes.

Note that a complete $302$-arc of Table 3 improves the result
of \cite{DGMP-JCD} for $q=2^{12}$.

From Tables 1-3 and \cite[Tab.\thinspace 1]{DFMP-JG}, we obtain
Theorem \ref{th2_Dq75}.
\begin{theorem}\label{th2_Dq75}
Let $173\le q\le 2879$ and $q=3511,4096$,
$4523,5003,5347,5641,$ $5843,6011$. Then
$$
t_{2}(2,q)<0.743\sqrt{q}\log_{2}^{0.75}q.
$$
\end{theorem}
Taking into account (\ref{eq2_Dq75}) and Table 3, we assume
that the following upper bound on the smallest size
$t_{2}(2,q)$ of complete arc in the plane $PG(2,q)$ holds.
\begin{conjecture}
It holds that
$$
t_{2}(2,q)<0.75\sqrt{q}\log_{2}^{0.75}q,\quad 173\le q.
$$
\end{conjecture}

\section{On the spectrum of possible sizes of complete arcs in $PG(2,q)$}
Let $m_{2}(2,q)$ be the greatest size of complete arcs in
$PG(2,q)$. For odd $q$, $m_{2}(2,q)=q+1$. For even $q$,
$m_{2}(2,q)=q+2$.   For $q=p^{2}$
 there is the complete $(q-\sqrt{q}+1)$-arc \cite{surveyhir}.
 For $q$ odd there is a complete $\frac{1}{2}(q+5)$-arc \cite{KorchSon2010}.
 For $q\equiv 2$ $(\bmod~3)$ odd, $11\le q\leq 3701$  \cite{DFMP-JG}, and for $q\equiv 1$
$(\bmod$~$4),$  $q\leq 337$ \cite{Giordano}, there is a
complete $\frac{1}{2}(q+7)$-arc. For even $q\ge8$ there is a
complete $\frac{1}{2}(q+4)$-arc \cite{librohir}.

For even $q$, let $M_{q}=\frac{1}{2}(q+4)$. For odd $q$, let
$M_{q}=\frac{1}{2}(q+7)$ if either $q\equiv 2$ $(\bmod~3)$,
$11\le q\leq 3701$, or $q\equiv 1$ $(\bmod$~$4)$, $q\leq 337$.
Else, $M_{q}=\frac{1}{2}(q+5)$.

Below we suppose that $\overline{t}_{2}(2,q)$ is given in
\cite[Tab.\thinspace 1]{DFMP-JG} for $q\le 841$, $q\neq343$,
and in Tables 1 and 2 of this paper for $853\le q\le 2879$.
Also, in this work we have obtained the value
$\overline{t}_{2}(2,343)=66$ that improves the result of
\cite{DFMP-JG}.

\begin{theorem}\label{Th_spectrum}
In $PG(2,q)$ with $25\le q\le 251$, $257\le q\le 349$, and
$q=1013,$ $2003$, there are complete $k$-arcs of \textbf{all}
the sizes in the region\\
 \centerline{$ \overline{t}_{2}(2,q)\le
k\le M_{q}.$}

\noindent In $PG(2,256)$ there are complete $k$-arcs of sizes
$k=55-123,130,241,258$.
\end{theorem}
\begin{proof}
For $25\le q\le 167$ the assertion follows from
\cite[Tab.\thinspace2]{DFMP-Plane} and \cite[Tab.\thinspace
2]{DFMP-JG}. For $169\le q\le 349$ and $q=1013,2003$, all the
results are obtained in this work by the randomized greedy
algorithms of \cite{DFMP-Plane,DMP-JG2004} with a new approach
to creation of starting conditions and data.
\end{proof}
\begin{conjecture}
Let $353\le q\le 2879$ be an odd prime. Then in $PG(2,q)$ there
are complete $k$-arcs of \textbf{all} the sizes in the region
$\overline{t}_{2}(2,q)\le k\le M_{q}$. Moreover, complete
$k$-arcs with $\overline{t}_{2}(2,q)\le k\le \frac{1}{2}(q+5)$
can be obtained by the randomized greedy algorithms of
\emph{\cite{DFMP-Plane,DMP-JG2004} } with a new approach to
creation of starting data.
\end{conjecture}

Our methods are applicable using our present computers
 for $q\le5171$. For reason of space we plan to write more complete
results and to describe this new approach to creation of
starting conditions and data in a journal paper.


\begin{thebibliography}{99}
\bibitem{DFMP-Plane}  A.\,A.~Davydov, G. Faina, S. Marcugini,
    and F. Pambianco, Computer search in projective planes for
    the sizes of complete arcs, \emph{J. Geom}., \emph{\
    }\textbf{82}, 50--62, 2005.

\bibitem{DFMP-JG}  A.\,A.~Davydov, G.~Faina, S.~Marcugini, and
    F.~Pambianco, On sizes of complete caps in projective
    spaces $PG(n,q)$ and arcs in planes $PG(2,q)$, \emph{J.
    Geom}., \textbf{94}, 31--58, 2009.

\bibitem{DGMP-JCD}  A.\,A.~Davydov, M.~Giulietti,
    S.~Marcugini, and F.~Pambianco, New inductive constructions
    of complete caps in $PG(N,q)$, $q$ even, \emph{J. Comb.
    Des.}, \textbf{18}, no. 3, 176--201, 2010.

\bibitem{DGMP-Innov}  A.\,A.~Davydov, M.~Giulietti,
    S.~Marcugini, and F.~Pambianco, On sharply transitive sets
    in $PG(2,q),$ \emph{Innov. Incid. Geom.},
    \textbf{6}-\textbf{7}, 139--151, 2009.

\bibitem{DMP-JG2004}  A.\,A. Davydov, S. Marcugini and F.
    Pambianco, Complete caps in projective spaces
    $\mathrm{PG}(n,q)$, \emph{J. Geom.,} \textbf{80\ } (2004)
    23--30.

\bibitem{Giordano} V.~Giordano, Arcs in cyclic affine planes,
    \emph{Innov. Incid. Geom.} \textbf{6}-\textbf{7},
    203--209, 2009.

\bibitem{G}  M.~Giulietti, Small complete caps in $PG(2,q)$ for
    $q$ an odd square, \emph{J. Geom.}, \textbf{69}, 110--116,
    2000.

\bibitem{Hartman-Haskin} A.~Hartman and L.~Raskin, Problems and
    algorithms for covering arrays, \emph{Discrete Math.,}
    \textbf{284}, no. 1, 149--156, 2004.

\bibitem{librohir}  J.\,W.\,P.~Hirschfeld, \textit{Projective
    geometries over finite fields}, Clarendon Press, Oxford,
    1998.

\bibitem{surveyhir}  J.\,W.\,P.~Hirschfeld and L.~Storme, The
    packing problem in statistics, coding theory and finite
    geometry: update 2001, in \textit{ Finite Geometries,
    Developments of Mathematics}, \textbf{3}, (Proc. of the Fourth Isle of
    Thorns Conf., Chelwood Gate, July 16-21, 2000),
    201--246, Eds. A.~Blokhuis, J.\,W.\,P.~Hirschfeld, D.~Jungnickel
    and J.\,A.~Thas, Kluwer, 2001.

\bibitem{Keri} G.~Keri, Types of superregular matrices and the
    number of $n$-arcs and complete $n$-arcs in $PG (r, q)$, \emph{J. Comb.
    Des.}, \textbf{14}, 363--390, 2006.

\bibitem{KV}  J.\,H.~Kim and  V.~Vu, Small complete arcs in
    projective planes, \emph{Combinatorica}, \textbf{23},
    311--363,  2003.

\bibitem{KorchSon2010} G.~Korchm\'{a}ros and A.~Sonnino, On
    arcs sharing the maximum number of points with an oval in a
    Desarguesian plane of odd order, \emph{J. Comb. Des.},
    \textbf{18}, 25--47, 2010.
\end{thebibliography}
\end{document}